\documentclass[10pt,a4paper]{amsart}

\usepackage[latin1]{inputenc}
\usepackage{amsmath}
\usepackage{amsthm}
\usepackage{amsfonts}
\usepackage{amssymb}
\usepackage{enumerate}
\usepackage{nicefrac}
\usepackage{graphicx}
\usepackage{hyperref}
\usepackage[all,cmtip]{xy}

\DeclareMathOperator{\poly}{poly}
\DeclareMathOperator{\orb}{orb}

\DeclareMathOperator{\Sing}{Sing}

\DeclareMathOperator{\ord}{ord}
\DeclareMathOperator{\GL}{GL}
\DeclareMathOperator{\Eh}{L}

\def\K{\kappa}
\def\C{\mathbb{C}}
\def\P{\mathbb{P}}
\def\qa{{\boldsymbol{{G}}}}
\def\cD{{D}}
\def\Diff{\Delta}
\def\Z{\mathbb{Z}}
\def\ZZ{\mathbb{Z}}

\def\CC{\mathcal{C}}
\def\Q{\mathbf{Q}}
\def\QQ{\mathbb{Q}}
\def\bQ{\mathbf{Q}}
\def\RR{\mathbb{R}}

\def\PP{\mathbb{P}}
\def\w{w}
\def\bw{\bar \omega}

\def\cO{\mathcal{O}}
\def\cC{\mathcal{C}}

\def\O{\mathcal{O}}
\def\cP{\mathcal{P}}

\def\red{{\rm red}}
\def\logres{{\rm LR}}

\newcommand\nul{\text{nul}}

\newtheorem{thm}{Theorem}[section]  %CHOOSE [chapter] or [section]
\newtheorem{prop}{Proposition}[section]
\newtheorem{cor}{Corollary}[section]
\newtheorem{lemma}{Lemma}[section]
\theoremstyle{remark}
\newtheorem{remark}{Remark}[section]
\newtheorem{notation}{Notation}[section]
\theoremstyle{definition}
\newtheorem{dfn}{Definition}[section]
\newtheorem{example}{Example}[section]

\def\leftblowup#1#2{\smash{\mathop{\longleftarrow}
\limits^{\array{ll} _{#1} \\ ^{#2} \endarray}}}

\makeatletter
\let\c@lemma\c@thm
\let\c@prop\c@thm
\let\c@conj\c@thm
\let\c@cor\c@thm
\let\c@rem\c@thm
\let\c@dfn\c@thm
\let\c@notation\c@thm
\let\c@exam\c@thm
\makeatother

\def\makeautorefname#1#2{\expandafter\def\csname#1autorefname\endcsname{#2}}

\makeautorefname{thm}{Theorem}%
\makeautorefname{lemma}{Lemma}%
\makeautorefname{rem}{Remark}%
\makeautorefname{notation}{Notation}%
\makeautorefname{dfn}{Definition}%
\makeautorefname{exam}{Example}%
\makeautorefname{prop}{Proposition}%
\makeautorefname{cor}{Corollary}%

\title
[Numerical Adjunction Formulas...]
{Numerical Adjunction Formulas for Weighted Projective Planes and Lattice Points Counting}

\author[J.I. Cogolludo]{Jos{\'e} Ignacio Cogolludo-Agust{\'i}n}
\address{Departamento de Matem\'aticas, IUMA\\ 
Universidad de Zaragoza\\ 
C.~Pedro Cerbuna 12\\ 
50009 Zaragoza, Spain} 
\email{jicogo@unizar.es} 

\author[J. Martín]{Jorge Martín-Morales}
\address{Centro Universitario de la Defensa-IUMA\\ 
Academia General
Militar\\ 
Ctra. de Huesca s/n.\\ 
50090, Zaragoza, Spain} 
\email{jorge@unizar.es,jortigas@unizar.es} 

\author[J. Ortigas]{Jorge Ortigas-Galindo}

\thanks{All authors are partially supported by
the Spanish Ministry of Education MTM2013-45710-C2-1-P and 
\emph{E15 Grupo Consolidado Geometr\'{\i}a} from the Gobierno de Aragón. 
The second author is also supported by FQM-333 from  
Junta de Andaluc{\'\i}a.
}

\keywords{Quotient surface singularity, invariants of curve singularities, Ehrhart polynomial, rational polytope}

\subjclass[2010]{32S05, 32S25, 52B20, 11F20}
\begin{document}

\begin{abstract}
This paper gives an explicit formula for the Ehrhart quasi-poly\-nomial of certain 2-dimensional polyhedra 
in terms of invariants of surface quotient singularities. Also, a formula for the dimension of the space of
quasi-homogeneous polynomials of a given degree is derived. This admits an interpretation as a 
Numerical Adjunction Formula for singular curves on the weighted projective plane.
\end{abstract}

\maketitle

\section{Introduction}
This paper deals with the general problem of counting lattice points in a polyhedron with 
rational vertices and its connection with both singularity theory of surfaces and Adjunction 
Formulas for curves in the weighted projective plane. 
In addition, we focus on rational polyhedra (whose vertices are rational points) as opposed to 
lattice polyhedra (whose vertices are integers).
Our approach exploits the connexion between Dedekind sums (as originated from the work of
Hirzebruch-Zagier~\cite{Hirzebruch-atiyah}) and the geometry of cyclic quotient singularities,
which has been proposed by several authors 
(see e.g.~\cite{pommersheim,Blache-RiemannRoch,Brion-points,laterveer,Urzua-arrangements,Buckley-icecream,
Ashikaga,Ashikaga-ishizaka}).

According to Ehrhart~\cite{Ehrhart-book}, the number of integer points of a lattice (resp.~rational) 
polygon $\cP$ and its dilations $d\cP=\{dp\mid p\in \cP\}$ is a polynomial (resp. quasi-polynomial) in $d$ 
of degree $\dim \cP$ referred to as the \emph{Ehrhart (quasi)-polynomial} of~$\cP$ (cf.~\cite{Chen-generalized}). 
In this paper we focus on the Ehrhart quasi-polynomial of polygons of type 
\begin{equation}\label{eq-dp}
d\cD_\w=\cD_{w,d}=\{(x,y,z)\in \RR^3\mid x,y,z\geq0, w_0x+w_1y+w_2z=d\},
\end{equation}
where $w=(w_0,w_1,w_2)$ are pairwise coprime and 
\begin{equation}\label{eq-dw}
\cD_\w:=\{(x,y,z)\in \RR^3\mid x,y,z\geq0, w_0x+w_1y+w_2z=1\}
\end{equation}
is a rational polygon.
In Theorem~\ref{AdjFor} we give an explicit formula for the Ehrhart quasi-polynomial 
of~\eqref{eq-dp}, which in Theorem~\ref{thmDELTAJ} is shown to be an invariant of the quotient 
singularities of the weighted projective plane $\PP^2_w$.

Throughout this paper $w_0,w_1,w_2$ are assumed to be pairwise coprime integers. 
Denote by $w=(w_0,w_1,w_2)$, $\bar{w}=w_0w_1w_2$, and
$|w|=w_0+w_1+w_2$. 
Finally, the key ingredients to connect the arithmetical problem referred to above with the 
geometry of weighted projective planes come from the observation that 
  \begin{equation}\label{cardinal}
\Eh_{\w}(d):= \# (\cD_{\w,d}\cap \ZZ^3)=h^0(\P^2_\w;\O(d)),
  \end{equation}
that is, the dimension of the vector space of weighted homogeneous polynomials of degree $d$,
and from a Numerical Adjunction Formula relating $h^0(\P^2_\w;\O(d))$ with the genus of a curve in~$\PP^2_\w$.

To explain what we mean by \emph{Numerical Adjunction Formulas}, assume a quasi-smooth curve 
$\CC\subset \PP^2_\w$ of degree $d$ exists. In that case, according to the classical Adjunction 
Formula one has the following equality relating canonical divisors on $\cC$ and $\PP_\w^2$
\begin{equation}\label{adclas}
K_{\CC}=(K_{\P^2_\w}+\CC)|_\CC.
\end{equation}
Equating degrees on both sides of~\eqref{adclas} and using the Weighted Bézout's Theorem, one has
$$2g(\CC)-2=\deg (K_{\P^2_\w}+\CC)|_\CC=\dfrac{\deg
(\CC)\deg(K_{\P^2_\w}+\CC)}{\bar{\w}}=\dfrac{d(d-|\w|)}{\bar{\w}}.$$

Notice that the generic curve of degree $k\bar w$ is smooth (see \cite[Lemma 5.4]{CMO12}). 
In that case, one has (c.f.~\S\ref{secex})
\begin{equation}\label{eq-adj-form}
h^0(\P^2_\w;\O(k\bar \w-|w|))=\Eh_{\w}(k\bar \w-|w|)=g_{\w,k\bar{\w}},
\end{equation}
where $g_{\w,t}:=\dfrac{t(t-|w|)}{2\bar{w}}+1$.

However, for a general $d$ the generic curve of degree $d$ in $\PP^2_\w$ is not necessarily 
quasi-smooth (see \cite[Lemma 5.4]{CMO12}).
The final goal of this paper is to revisit~\eqref{eq-adj-form} in the general (singular) case.

Let us present the main results of this work. The first main statement shows an explicit formula
for the Ehrhart quasi-polynomial $\Eh_\w(d)$ of degree two of $\cD_{\w,d}$ in terms of~$d$.

\begin{thm}\label{AdjFor}
The Ehrhart quasi-polynomial $\Eh_\w(d)$ for the polygon $\cD_\w$ 
in~\eqref{eq-dw} satisfies
$$\Eh_\w(d)=g_{\w,d+|\w|}-\sum_{P\in\Sing(\P^2_w)}\Diff_P(d+|\w|).$$
\end{thm}

The quadratic polynomial $g_{\w,k}=\dfrac{k(k-|w|)}{2\bar{w}}+1$ in $k$ is called 
the \emph{virtual genus} (see \cite[Definition 5.1]{CMO12}) and $\Diff_P(k)$
is a periodic function of period $\bar \w$ which is an invariant associated to the singularity 
$P\in \Sing(\P^2_w)$ (see Definition~\ref{defDELTA}).

The proof of Theorem~\ref{AdjFor} relies heavily on computations with Dedekind sums.

The next result aims to show that the previous combinatorial number $\Diff_P(k)$ has a geometric
interpretation and can be computed via invariants of curve singularities on a singular surface. 
In order to do so we recall the recently defined invariant $\delta_P(f)$ of a curve $(\{f=0\},P)$ on a 
surface with quotient singularity (see \cite[Section 4.2]{CMO12})
and we define a new invariant $\K_P(f)$ in \S\ref{secccondiciones}.

\begin{thm}\label{thmDELTAJ}
Let $(f,P)$ be a reduced curve germ at $P\in X$ a surface cyclic quotient singularity.
Then 
$$\Diff_P(k)= \delta_P(f)-\K_P(f)$$
for any reduced germ $f\in \O_{X,P}(k)$.
\end{thm}
The module $\cO_{X,P}(k)$ of $k$-invariant germs of $X$ at $P$ can be found in Definition~\ref{deforbi}.

As an immediate consequence of Theorems~\ref{AdjFor} and \ref{thmDELTAJ} one has a method to compute 
$\Eh_\w(d)$ by means of appropriate curve germs $(\{f=0\},P)$ on surface quotient singularities.
In an upcoming paper we will study the $\Diff_P(k)$-invariant by means of singularity theory and 
intersection theory on surface quotient singularities in order to give a closed effective formula 
for the Ehrhart quasi-polynomial $\Eh_\w(d)$.
In fact, Theorem~\ref{AdjFor} can also be seen as a version of~\cite{Blache-RiemannRoch}, where an 
explicit interpretation of the \emph{correction term} is given in Theorem~\ref{thmDELTAJ}.

Finally, we generalize the Numerical Adjunction Formula for a general singular curve $\cC$ on $\PP^2_\w$
relating $h^0(\PP^2,\cO_{\PP^2_\w}(d-|w|))$, its genus $g(\cC)$, and the newly defined invariant $\K_P$.

\begin{thm}[Numerical Adjunction Formula]\label{Adj-like-form}
Consider $\cC=\{f=0\}\subset \PP^2_{w}$ an irreducible curve of degree $d$, then
$$h^0(\PP^2_\w,\cO_{\PP^2_\w}(d-|w|))=g(\CC)+\sum _{P\in\Sing (\cC)} \K_P(f).$$
\end{thm}

This paper is organized as follows. In \S\ref{secInt} some basic definitions
and preliminary results on surface quotient singularities, logarithmic forms, and Dedekind sums
are given. In \S\ref{secdelta}, after defining the three local invariants mentioned above, 
a proof of Theorem~\ref{thmDELTAJ} is given. An introductory example is treated in \S\ref{secex} and finally, 
the main results Theorem~\ref{AdjFor} and Theorem~\ref{Adj-like-form} are proven in \S\ref{Goretti}.

\section{Definitions and Preliminaries}\label{secInt}
In this section some needed definitions and results are provided.
\subsection{\texorpdfstring{$V$}{V}-manifolds and Quotient Singularities}

We start giving some basic definitions and properties of $V$-manifolds, weighted projective spaces, embedded $\Q$-resolutions, 
and weighted blow-ups (for a detailed exposition see for instance
\cite{Dolgachev82,AMO11a,AMO11b,Martin11PhD,Ortigas13PhD}).
Let us fix the notation and introduce several tools to calculate a special kind
of embedded resolutions, called \emph{embedded $\Q$-resolutions} (see
Definition~\ref{Qresolution}), for which the ambient space is allowed to contain
abelian quotient singularities.
To do this, we study weighted blow-ups at points.

\begin{dfn}
A $V$-manifold of dimension $n$ is a complex analytic space which admits an open covering $\{U_i\}$ such that $U_i$ is analytically isomorphic to $B_i/G_i$ where $B_i \subset \C^n$ is an open ball and $G_i$ is a finite subgroup of $\GL(n,\C)$.
\end{dfn}

We are interested in $V$-surfaces where the quotient spaces $B_i / G_i$ are given by (finite) abelian groups.

Let $\qa_{d} \subset \C^{*}$ be the cyclic group of $d$-th roots of unity generated by $\xi_d$. Consider a vector of
weights $(a,b)\in\Z^2$ and the action
\begin{equation}\label{actiondab}
\begin{array}{ccc}
 \qa_d \times \C^2 & \overset{\rho}{\longrightarrow} & \C^2, \\
(\xi_d, (x,y)) & \mapsto & (\xi_d^{a}\, x,\xi_{d}^{b}\, y).
\end{array}
\end{equation}
The set of all orbits $\C^2 / \qa_{d}$ is called a {\em cyclic quotient
space of type $(d;a,b)$} and it is denoted by $X(d;a,b)$.

The type $(d;a,b)$ is \emph{normalized} if and only if $\gcd(d,a) = \gcd(d,b) =
1$. If this is not the case, one uses the isomorphism (assuming $\gcd(d,a,b)=1$)
$$
\begin{array}{rcl}
X(d; a,b)  & \longrightarrow & X \left( \frac{d}{(d,a)(d,b)}; \frac{a}{(d,a)}, \frac{b}{(d,b)} \right), \\[0.3cm]
\big[ (x,y) \big] & \mapsto & \big[ (x^{(d,b)},y^{(d,a)}) \big]
\end{array}
$$
to normalize it.

We present different properties of some important sheaves associated to a $V$-surface (see  \cite[\S$4$]{sasakian} and \cite{Dolgachev82}).

\begin{prop}[\cite{sasakian}]\label{proporbi}
 Let $\O_X$ be the structure sheaf of a $V$-surface $X$ then,
 
 \begin{itemize}
  \item If $P$ is not a singular point of $X$ then $\O_{X,P}$ is isomorphic to
the ring of convergent power series $\C\{x,y\}$.
  \item If $P$ is a singular point of $X$ then $\O_{X,P}$ is isomorphic to the
ring of $\qa_d$-invariant convergent power series $\C\{x,y\}^{\qa_d}$.
 \end{itemize}
\end{prop}

If no ambiguity seems likely to arise we simply write $\O_{P}$ for the
corresponding local ring or just $\O$ in the case $P=0$.
\begin{dfn}\label{deforbi}
Let $\qa_{d}$ be an arbitrary finite cyclic group, a vector of weights $(a,b)\in\Z^2$ and the 
action given in (\ref{actiondab}).
Associated with $X(d;a,b)$ one has the following $\O_{X,P}$-module: 
$$\O_{X,P}(k):=\{ h \in \C\{x,y\}|\ h(\xi_d^ax,\xi_d^by)=\xi_d^k h(x,y)\},$$
also known as the module of \emph{$k$-invariant} germs in~$X(d;a,b)$.
\end{dfn}

\begin{remark}
Note that 
\begin{equation}
\label{eq-descomp-cxy}
\C\{x,y\}=\bigoplus_{k=0}^{d-1}\O_{X,P}(k)
\end{equation}
\end{remark}

\begin{remark}[\cite{sasakian}]
 Let $l,k\in \Z$. Using the notation above one clearly has the following properties: 
 \begin{itemize}
 \item $\O_{X,P}(k)=\O_{X,P}(d+k),$
 \item $\O_{X,P}(l)\otimes \O_{X,P}(k)\subset \O_{X,P}(l+k).$
 \end{itemize}
\end{remark}

These modules produce the corresponding sheaves $\O_X(k)$ on a $V$-surface $X$ also called \emph{orbisheaves}.

One of the main examples of $V$-surfaces is the so-called \emph{weighted
projective plane} (e.g.~\cite{Dolgachev82}). Let $w:=(\w_0,\w_1,\w_2)\in \Z_{>0}^3$
be a weight vector, that is, a triple of pairwise coprime positive
integers. 
There is a natural action of the multiplicative group $\C^{*}$ on
$\C^{3}\setminus\{0\}$ given by
$$
  (x_0,x_1,x_2) \longmapsto (t^{\w_0} x_0,t^{\w_1} x_1,t^{\w_2} x_2).
$$
The universal geometric quotient of $\frac{\C^{3}\setminus\{0\}}{\C^{*}}$ 
under this action is denoted by $\PP^2_w$ and it is called the {\em weighted projective plane} of type $\w$.

Let us recall the adapted concept of resolution in this category.

\begin{dfn}[\cite{Martin11PhD}]\label{Qresolution}
An {\em embedded $\Q$-resolution} of a hypersurface $(H,0) \subset (M,0)$ in an abelian quotient
space is a proper analytic map~$\pi: X \to (M,0)$ such that:
\begin{enumerate}
\item $X$ is a $V$-manifold with abelian quotient singularities,
\item $\pi$ is an isomorphism over $X\setminus \pi^{-1}(\Sing(H))$,
\item $\pi^{-1}(H)$ is a $\Q$-normal crossing hypersurface on $X$ (see
\cite[Definition~1.16]{Steenbrink77}).
\end{enumerate}
\end{dfn}

Embedded $\Q$-resolutions are a natural generalization of the usual embedded
resolutions, for which some invariants, such as $\delta$ can be effectively
calculated~(\cite{CMO12}).

As a key tool to construct embedded $\Q$-resolutions of abelian quotient surface singularities we will recall 
toric transformations or weighted blow-ups in this context (see~\cite{Oka-nondegenerate} as a general reference),
which can be interpreted as blow-ups of $\mathfrak{m}$-primary ideals.

Let $X$ be an analytic surface with abelian quotient singularities. Let us define the weighted blow-up 
$\pi: \widehat{X} \to X$ at a point $P\in X$ with respect to $\w = (p,q)$. 
Since it will be used throughout the paper, we briefly describe the local equations of a weighted blow-up at a
point $P$ of type $(d;a,b)$ (see \cite[Chapter 1]{Martin11PhD} for further details).

The birational morphism $\pi = \pi_{(d;a,b),\w}: \widehat{X(d;a,b)}_{\w} \to X(d;a,b)$ can be described as usual by 
covering $\widehat{X(d;a,b)}_{\w}$ into two charts $\widehat{U}_1 \cup \widehat{U}_2$, where for instance 
$\widehat{U}_1$ is of type $X \left( \displaystyle\frac{pd}{e}; 1, \frac{-q+a' pb}{e} \right)$, 
with $a'a=b'b\equiv 1 \mod (d)$ and $e = \gcd(d,pb-qa)$. The first chart is given by
\begin{center}
\begin{equation}
\label{eq-charts}
\begin{array}{rcl}
X \left( \displaystyle\frac{pd}{e}; 1, \frac{-q+a' pb}{e} \right)  & \longrightarrow &
\widehat{U}_1, \\ 
\big[ (x^e,y) \big] & \mapsto &  
\big[ ((x^p,x^q y),[1:y]_{\w}) \big]_{(d;a,b)}
\end{array}
\end{equation}
\end{center}
The second one is given analogously.

The exceptional divisor $E = \pi_{(d;a,b),\w}^{-1}(0)$ is identified with 
$\PP^1_{\w}(d;a,b) := \PP^1_{\w}/\qa_{d}$. The singular points are cyclic 
and correspond to the origins of the two charts.

\subsection {Log-resolution logarithmic forms}\label{resvmanifolds}
All the preliminaries about De Rham cohomology for projective varieties
with quotient singularities can be found in~\cite[Chapter 1]{Steenbrink77} and
the ones about $C^\infty$~log complex of quasi projective algebraic varieties in
\cite[\S$1.3$]{JIphd}.
Here we focus on the non-normal crossing $\QQ$-divisor case in
weighted projective planes.

Let $\cD$ be a  $\QQ$-divisor in a surface $X$ with quotient singularities. The
complement of $\cD$ will be denoted by $X_{\cD}$. 
Let us fix
$\pi : Y  \longrightarrow X$
a $\Q$-resolution of the singularities of $\cD$ so that
the reduced  $\QQ$-divisor
$\overline{\cD}=\pi^*(\cD)_{\red}$ is a union of smooth  $\QQ$-divisors
on $Y$ with $\Q$-normal crossings.

Using the results in~\cite{Steenbrink77} we can generalize Definition 2.7 in
\cite{JIphd} for a non-normal crossing $\QQ$-divisor in~$X$.
 
\begin{dfn}\label{deflog}
% \begin{enumerate}
% \item
The sheaf $\pi_*{\Omega}_{Y}(\log \langle \overline{\cD} \rangle)$
is called the sheaf of \emph{log-resolution logarithmic forms} on $X$ 
with respect to~$\cD$.

\begin{remark}
Note that the space $Y$ in the previous definition is not smooth and we use 
the standard definition of logarithmic sheaf for $V$-manifolds and $V$-normal crossing divisors 
$\overline{\cD}$ due to Steenbrink~\cite{Steenbrink77}.
\end{remark}

In the sequel, a log-resolution logarithmic form with respect to a  $\QQ$-divisor 
$\cD$ and a $\Q$-resolution $\pi$ will be referred to as simply a logarithmic
form, if $\cD$ and $\pi$ are known and no ambiguity is likely to arise.
\end{dfn}

\begin{remark}\label{problema1forma}
Let $h$ be an analytic germ on $X(d;a,b)$ where the type is normalized. Notice
that $\omega=h\dfrac{dx\wedge dy}{xy}$ automatically defines a logarithmic form 
with poles along $xy$, whereas expressions of the form 
$\tilde{\omega}=h\dfrac{dx\wedge dy}{x}$ might not even define a $2$-form
unless $h$ is so that $\tilde{\omega}$ is invariant under~$\qa_d$. 
\end{remark}

Also note that $\pi_*{\Omega}_{Y}(\log \langle \overline{\cD} \rangle)$ depends, in principle,
on the given resolution $\pi$. The following results shows that this is not the case.

\begin{prop}\label{nodepende} 
The sheaves 
$\pi_* {\Omega^\bullet}_{Y} (\log \langle \overline{D} \rangle)$
of logarithmic forms on $X$ with respect to the $\QQ$-divisor $\cD$
do not depend on the chosen $\Q$-resolution.
\end{prop}

\begin{proof}
Let $Y$ and $Y'$ be two $\Q$-resolutions of $(X,\cD)$. After resolving $(Y,\bar \cD)$ and 
$(Y',\bar \cD')$ and applying the strong factorization theorem for smooth surfaces, there 
exists a smooth surface $\tilde Y$ obtained as a finite number of blow-ups of both $Y$ and $Y'$
which is a common resolution of $(Y,\bar \cD)$ and $(Y',\bar \cD')$. 
Since 
$$\Omega^\bullet_{Y} (\log \langle \overline{D} \rangle)=
\rho_*\Omega^\bullet_{\tilde Y} (\log \langle \tilde{D} \rangle) 
\quad
\text{and} 
\quad
\Omega^\bullet_{Y'} (\log \langle \overline{D'} \rangle)=
\rho'_*\Omega^\bullet_{\tilde Y} (\log \langle \tilde{D} \rangle)$$
(see~\cite[p.~351]{Steenbrink77}) where $\overline{D},\overline{D'},$ and $\tilde{D}$ are the 
corresponding total preimages and $\rho$, $\rho'$ are the corresponding resolutions.
The result follows, since 
$$
\pi_* \Omega^\bullet_{Y} (\log \langle \overline{D} \rangle)=
\pi_*\rho_*\Omega^\bullet_{\tilde Y} (\log \langle \tilde{D} \rangle)=
\pi'_*\rho'_*\Omega^\bullet_{\tilde Y} (\log \langle \tilde{D} \rangle)=
\pi'_* \Omega^\bullet_{Y'} (\log \langle \overline{D'} \rangle)
$$
due to the commutativity of the diagram $\pi\rho=\pi'\rho'$.
\end{proof}

\begin{notation}
\label{not-logres}
In the future, we will refer to such sheaves as logarithmic sheaves on $\cD$ and they will 
be denoted simply as
$
\Omega^\bullet_{X} (\logres \langle \cD \rangle).
$
\end{notation}

\subsection{Dedekind Sums}\label{dedekind_sec}

Let $a,b,c,t$ be positive integers with $\gcd(a,b)=\gcd(a,c)=\gcd(b,c)=1$. The
aim of this part (see \cite{Beck07} for further details) is to
give a way to compute the cardinal of the following two sets:
$$\Delta_1:=\{(x,y)\in \Z_{\geq 0}^2|\ ax+by\leq t\},$$
$$\Delta_2:=\{ (x,y,z)\in \Z_{\geq 0}^3|\ ax+by+cz=t\}.$$

Note that $\#\Delta_1$ cannot be computed by means of Pick's Theorem unless $t$
is divisible by $a$ and $b$.

Denote by $L_{\Delta_i}(t)$ the cardinal of $\Delta_i$. Let us consider the
following notation.

\begin{notation}
If we denote by $\xi_a:=e^{\frac{2i\pi}{a}}$, consider

 \begin{equation}
  \label{ecpw*}
\array{lcl}
p_{\{a,b,c\}}(t)&:=&\poly_{\{a,b,c\}}(t)+\frac{1}{a}\sum_{k=1}^{a-1}\frac{1}{
(1-\xi_a^{kb})(1-\xi_a^{kc})\xi_a^{kt}} \\
\\
&+&\frac{1}{b}\sum_{k=1}^{b-1}\frac{1}{(1-\xi_b^{ka})(1-\xi_b^{kc})\xi_b^{kt}}
+\frac{1}{c}\sum_{k=1}^{c-1}\frac{1}{(1-\xi_c^{ka})(1-\xi_c^{kb})\xi_c^{kt}},
\endarray
\end{equation}
 with
  \begin{equation*}
%   \label{dim3}
 \poly_{\{a,b,c\}}(t):=\frac{t^2}{2abc}+\frac{t}{2}\left(\frac{1}{ab}+\frac{1}{
ac}+\frac{1}{bc}\right)+\frac{3(ab+ac+bc)+a^2+b^2+c^2}{12abc}.
\end{equation*}
\end{notation}

\begin{remark}
 Notice that in particular, one has
\begin{equation}\label{ecpw1*}  
p_{\{a,b,1\}}(t)=\poly_{\{a,b,1\}}(t)+\frac{1}{a}\sum_{k=1}^{a-1}\frac{1}{
(1-\xi_a^{kb})(1-\xi_a^{k})\xi_a^{kt}}+\frac{1}{b}\sum_{k=1}^{b-1}\frac{1}{
(1-\xi_b^{ka})(1-\xi_b^{k})\xi_b^{kt}},
\end{equation}
with
\begin{equation*}
 \poly_{\{a,b,1\}}(t)=\frac{t^2}{2ab}+\frac{t}{2}\left(\frac{1}{ab}+\frac{1}{a}
+\frac{1}{b}\right)+\frac{3(ab+a+b)+a^2+b^2+1}{12ab}.
\end{equation*}
\end{remark}

\begin{thm}[\cite{Beck07}]
 One has the following result,
 $$L_{\Delta_1}(t)=p_{\{a,b,1\}}(t) \text{ and }
L_{\Delta_2}(t)=p_{\{a,b,c\}}(t).$$
\end{thm}
 
Now we are going to define the Dedekind sums giving some properties
which will be particularly useful for future results. See~\cite{Rad72} and
\cite{Beck07} for a more detailed exposition.

\begin{dfn}[\cite{Rad72}]\label{dedekind}
Let $a,b$ be integers, $\gcd(a,b)=1$, $b\geq 1$. The Dedekind sum $s(a,b)$ is
defined as follows
\begin{equation}\label{lasuma}
 s(a,b):=\sum_{j=1}^{b-1}\left(\left(\frac{ja}{b}\right)\right)\left(\left(\frac
{j}{b}\right)\right),
\end{equation}
where the symbol $\left(\left(x\right)\right)$ denotes
$$ \left(\left(x\right)\right) = \left\{ \begin{array}{ll}
				    x-[x]-\frac{1}{2} & \mbox{if $x$ is not an
integer},\\
				    0 & \mbox{if $x$ is an integer},\end{array}
\right.$$
with $[x]$ the greatest integer not exceeding $x$. 
\end{dfn}

The following result, referred to as a Reciprocity Theorem (see \cite[Corollary
$8.5$]{Beck07} or  \cite[Theorem $2.1$]{Rad72} for further details) will be key in 
what follows.

\begin{thm}[Reciprocity Theorem, \cite{Beck07,Rad72}]\label{prop4}
Let $a$ and $b$ be two coprime integers. Then 
$$s(a,b)+s(b,a)=-\frac{1}{4}+\frac{1+a^2+b^2}{12ab}.$$
\end{thm}

Let us express the sum (\ref{lasuma}) in terms of $a$th-roots of the unity (see
for instance \cite[Example $8.1$]{Beck07} or \cite[Chapter $2$, $(18b)$]{Rad72}
for further details).

\begin{prop}[\cite{Beck07,Rad72}]\label{prop3}
 Let $a,b$ be integers, $\gcd(a,b)=1$, $b\geq 1$, denote by $\xi_b$ a primitive
$b$th-root of the unity. The Dedekind sum $s(a,b)$ can be written as follows:
$$s(a,b)=\frac{b-1}{4b}-\frac{1}{b}\sum_{k=1}^{b-1}\frac{1}{(1-\xi_b^{ka}
)(1-\xi_b^k)}.$$
\end{prop}

Let us exhibit some useful properties of the Dedekind sum $s(a,b)$.

Since $((-x))=-((x))$ it is clear that

\begin{equation*}\label{prop6}
 s(-a,b)=-s(a,b)
\end{equation*}
 and also
$$s(a,-b)=s(a,b).$$
If we define $a'$ by $a'a\equiv 1 \mod b$ then 
$$s(a',b)=s(a,b).$$
\begin{prop}[\cite{Beck07,Rad72}]\label{prop5}
 Let $a,b,c$ be integers with $\gcd(a,b)=\gcd(a,c)=\gcd(b,c)=1$. Define $a'$ by
$a'a\equiv 1 \mod b$, $b'$ by $b'b\equiv 1 \mod c$ and $c'$ by $c'c\equiv 1 \mod
a$. Then
$$s(bc',a)+s(ca',b)+s(ab',c)=-\frac{1}{4}+\frac{a^2+b^2+c^2}{12abc}.$$
\end{prop}

\begin{dfn}[\cite{Beck07}]\label{FDS}
 Let $a_1,\ldots,a_m, n\in \ZZ$, $b\in\Z_{>0}$, then the Fourier-Dedekind sum is
defined as follows:
$$s_n(a_1,\ldots,a_m;b):= \frac{1}{b}\sum_{k=1}^{b-1}\frac{\xi_
b^{kn}}{(1-\xi_b^{ka_1})(1-\xi_b^{ka_2})\cdots(1-\xi_b^{ka_m})}.$$
\end{dfn}

Let us see some interesting properties of these sums.

\begin{remark}[\cite{Beck07}]\label{remDed}
Let $a,b,c\in \ZZ$ then

 \begin{enumerate}
\item  For all $n\in\ZZ$, $s_n(a,b;1)=0.$
\item  For all $n\in\ZZ$, $s_n(a,b;c)=s_n(b,a;c).$
\item One has $s_0(a,1;b)=-s(a,b)+\frac{b-1}{4b}.$
\item If we denote by $a'$ the inverse of $a$ modulo $c$, then $s_0(a,b;c)=
-s(a'b,c)+\frac{c-1}{4c}.$
 \end{enumerate}
\end{remark}

With this notation we can express (\ref{ecpw*}) and (\ref{ecpw1*}) as follows
\begin{equation}\label{ecpw1}
 p_{\{a,b,1\}}(t)= \poly_{\{a,1,b\}}(t)+s_{-t}(a,1;b)+s_{-t}(1,b;a).
\end{equation}
\begin{equation}\label{ecpw}
 p_{\{a,b,c\}}(t)=
\poly_{\{a,b,c\}}(t)+s_{-t}(a,b;c)+s_{-t}(b,c;a)+s_{-t}(a,c;b).
\end{equation}

 As a consequence of Zagier reciprocity in dimension $3$ (see \cite[Theorem
$8.4$]{Beck07}) one has the following result.
\begin{cor}[Rademacher's Reciprocity Law, \cite{Beck07}]\label{radrec}

Substituting $t=0$ in the previous expression one gets
$$1-\poly_{\{a,b,c\}}(0)=s_{0}(a,b;c)+s_{0}(c,b;a)+s_{0}(a,c;b)=-\frac{1}{4}
+\frac{a^2+b^2+c^2}{12abc}.$$
\end{cor}

\section{Local algebraic invariants on quotient singularities}\label{secdelta}

In this section we study two local invariants of a curve in a
$V$-surface, the delta invariant $\delta_P(\cC)$ and the dimension $\K_P(\cC)$
(see Definitions~\ref{deltaintro} and \ref{condicionesnul}) and a local
invariant of the surface, $\Diff_P(k)$ (see
Definition~\ref{defDELTA}), as well as the relation among them (see
Theorem~\ref{thmDELTAJ}), so as to understand the right-hand side of the
formula in Theorem~\ref{AdjFor}.

In \cite{CMO12, Ortigas13PhD} we started extending the concept of Milnor fiber
and Milnor number
of a curve singularity 
allowing the ambient space to be a quotient surface singularity. A generalization of the local 
$\delta$-invariant is also defined and described in terms of a $\bQ$-resolution of the curve singularity. 
All these tools allow for an explicit description of the genus formula of a curve defined on a weighted projective 
plane in terms of its degree and the local type of its singularities.

\begin{dfn}[\cite{CMO12}]
Let $\cC=\{f=0\} \subset X(d;a,b)$ be a curve germ. 
The \emph{Milnor fiber} $f_t^\w$ of $(\cC,[0])$ is defined as follows,
$$f_t^\w:=\{ f=t\}/{\qa_d}.$$
The \emph{Milnor number} $\mu^\w$ of $(\cC,P)$ is defined as follows,
$$\mu^\w:=1-\chi^{\orb}(f_t^\w).$$
\end{dfn}

We recall that $\chi^{\orb}(O):=\frac{1}{|G|}\sum_{\Delta} (-1)^{\dim \Delta} |G_\Delta|$
for an orbifold $O$ with a finite $CW$-complex structure given by the cells $\Delta$ and the finite group $G$ acting on it, 
where $G_\Delta$ denotes the stabilizer of~$\Delta$.

Note that alternative generalizations of Milnor numbers can be found, for instance, 
in~\cite{ABLM-milnor-number,brasselet-milnor,Le-Someremarks,STV-Milnornumbers}.
The one proposed here seems more natural for quotient singularities, 
but more importantly, it allows for the existence of an explicit formula relating Milnor number, $\delta$-invariant, 
and genus of a curve on a singular surface.

We define the local invariant $\delta$ for curve singularities on $X(d;a,b)$.

\begin{dfn}[\cite{CMO12}]\label{deltaintro}
Let $\cC$ be a reduced curve germ at $[0]\in X(d;a,b)$, then we define 
$\delta$ (or $\delta_0(\cC)$) as the number verifying 
\begin{equation*}
\chi^{\orb}(f_t^\w)=r^\w - 2 \delta,
\end{equation*}
where $r^\w$ is the number of irreducible branches of $\cC$ at $[0]$.
\end{dfn}

\begin{remark}
Note that the $\delta$-invariant of a reduced curve (i.e. a reduced $\QQ$-divisor) on a surface with quotient 
singularity is not an integer number in general, but rather a rational number (see~\cite[Example~4.6]{CMO12}). 
However, in the case when $\cC$ is in fact Cartier, $\delta_0(\cC)$ is an integer number that has an 
interpretation as the dimension of the quotient $\bar R/R$ where $R$ is the coordinate ring of $\cC$ and 
$\bar R$ its normalization (see~\cite[Theorem~4.14]{CMO12}).
An alternative definition of $\delta_0(\cC)$ on normal surfaces in terms of a resolution can be 
found in~\cite{Blache-RiemannRoch}. 

Also note that $r^\w$ can also be seen as the number of irreducible $k$-invariant factors 
of the defining equation $f$. For instance, the germ defined by $f=(x^2-y^4)$ in $\CC^2$ is not irreducible
since $(x^2-y^4)=(x-y^2)(x+y^2)$. However, $f=0$ also defines a set of zeroes in $X(2;1,1)$, which is 
irreducible (and hence $r^\w=1$), since $(x-y^2)$ and $(x+y^2)$ are not $k$-invariant for any $k$
(recall Definition~\ref{deforbi}).
\end{remark}

A recurrent formula for $\delta$ based on a $\Q$-resolution
of the singularity is provided in Theorem~\ref{thm-delta}.

Assume $(f,0)\subset X(d;a,b)$ and consider 
a $(p,q)$-blow-up $\pi$ at the origin. Denote by $\nu_0(f)$ the $(p,q)$-multiplicity 
of $f$ at $0$ and $e:=\gcd(d,pb-qa)$. As an interpretation of $\nu=\nu_0(f)$, we recall that 
$\pi^*(C)=\hat C+\frac{\nu}{e} E$, where $\pi^*(C)$ is the total transform of $C$,
$\hat C$ is its strict transform and $E$ is the exceptional divisor.

We will use the following notation:
\begin{equation}\label{eq-deltawpi}
 \delta_{0,\pi}(f)=\frac{\nu_0(f)}{2dpq}\left(\nu_0(f)-p-q+e \right),
\end{equation}

\begin{thm}[\cite{CMO12}]
\label{thm-delta}
Let $(C,[0])$ be a curve germ on an abelian quotient surface singularity. Then
\begin{equation}
\label{eq-delta}
\delta(C)=\sum_{Q\prec {[0]}} \delta_{Q,\pi_{(p,q)}}(f)
\end{equation}
where $Q$ runs over all the infinitely near points of a $\Q$-resolution of
$(C,{[0]})$ and
$\pi_{(p,q)}$ is a $(p,q)$-blow-up of $Q$, the origin of $X(d;a,b)$.
\end{thm}

\subsection{Logarithmic Modules}\label{secccondiciones}
\mbox{}

For a given $k\geq0$, one has the module $\O_P(k)$ of $k$-invariant germs (see Definition~\ref{deforbi}),
$$\O_P(k):=\{ h \in \C\{x,y\}|\ h(\xi_d^ax,\xi_d^by)=\xi_d^k h(x,y)\}.$$
Let $\{f=0\}$ be a germ in $X(d;a,b)$. Note that if $f\in \O_P(k)$, then one
has the following  $\O_P$-module
 $$\O_P(k-a-b)=\{ h \in \C\{x,y\}\mid\ h \frac{dx\wedge dy}{f} \text{ is }
\qa_{d} \text {-invariant}\}.$$

\begin{dfn} \label{modulitos1}
Let $\cD=\{f=0\}$ be a germ in $P\in X(d;a,b)$, where $f\in \O_P(k)$. 

\begin{enumerate}
\item
Let ${\mathcal M}^{\logres}_{\cD}$ denote the submodule of $\O_P(k-a-b)$
consisting of all $h \in \O_P(k-a-b)$ such that the 2-form
\begin{equation*}
\omega=h \ \frac{dx \wedge dy}{f}\in \Omega^2_X(\logres \langle D\rangle)
\end{equation*}
(recall Notation~\ref{not-logres}).

\item
Let ${\mathcal M}^{\nul}_{\cD}$ denote the submodule of ${\mathcal M}^{\logres}_{\cD}$
consisting of all $h \in {\mathcal M}^{\logres}_{\cD}$ such that the 2-form
\begin{equation*}
\omega=h \ \frac{dx \wedge dy}{f}\in \Omega^2_X(\logres \langle D\rangle)
\end{equation*}
admits a holomorphic extension outside the strict transform~$\widehat{f}$.
\item Any $\O_P$-module $\mathcal{M}\subseteq{\mathcal M}^{\logres}_{\cD}$ will be
called \emph{logarithmic module}.
\end{enumerate}
\end{dfn}

\begin{dfn}\label{condicionesnul}
Let $\cD=\{f=0\}$ be a germ in $P\in X(d;a,b)$. Let us define the following dimension,
$$\K_P(\cD)=\K_P(f):=\dim_{\C} \frac{\O_P(s)}{{\mathcal M}^{\nul}_{\cD}},$$
for $s=\deg f-a-b$.
\end{dfn}

\begin{remark}
From the discussion in \S\ref{resvmanifolds} note that $\K_P(f)$ turns out to be a finite number 
independent on the chosen $\Q$-resolution. Intuitively, the number $\K_P(f)$ provides the minimal 
number of conditions required for a generic germ $h\in \O_P(s)$ so that $h\in {\mathcal M}^{\nul}_{\cD}$.
\end{remark}

\begin{remark}\label{condicionesclasico}
It is known (see \cite[Chapter $2$]{JIphd}) that if $f$ is a holomorphic germ in $(\C^2,0)$, then $$\K_0(f)=\delta_0(f).$$
\end{remark}

\subsection{The \texorpdfstring{$\delta$}{delta} invariant in the general case of local germs}
\mbox{}

Let us start with the following constructive result which allows one to see any singularity 
on the quotient $X(d;a,b)$ as the strict transform of some $\{g=0\} \subset \C^2$ after 
performing a certain weighted blow-up.

\begin{remark}\label{weierstrass_remark}
The Weierstrass division theorem states that given $f,g \in \C\{x,y\}$ with $f$ $y$-general of order $k$, 
there exist $q \in \C\{x,y\}$ and $r \in \C\{x\}[y]$ of degree in $y$ less than or equal to $k-1$, both 
uniquely determined by $f$ and $g$, such that $g = q f + r$. The uniqueness and the linearity of the action 
ensure that the division can be performed equivariantly for the action of $\qa_d$ on $\C\{x,y\}$ (see~\eqref{eq-descomp-cxy}), 
i.e.~if $f,g\in \O(l)$, then so are $q$ and~$r$. In other words, the Weierstrass preparation theorem still holds for zero 
sets in~$\C\{x,y\}^{\qa_d}$.
\end{remark}

Let $\{f=0\}\subset (X(d;a,b),0)$ be a reduced analytic germ. Assume $(d;a,b)$ is a normalized type. 
After a suitable change of coordinates of the form $X(d;a,b)\to X(d;a,b)$, 
$[(x,y)]\mapsto [(x+\lambda y^k,y)]$ where $ bk\equiv a \mod d$, one can assume $x\nmid f$. 
Moreover, by Remark~\ref{weierstrass_remark}, $f$ can be written in the form
\begin{equation}\label{special_form2}
f(x,y) = y^r + \sum_{i> 0, \ j<r} a_{ij} x^i y^j \ \in \ \C\{x\}[y]\cap \cO(k).
\end{equation}
For technical reasons, in the following results the space $X(d;a,b)$ will be considered to be of type
$X(p;-1,q)$. Note that this is always possible.

\begin{lemma}\label{lema1b}
Let $f \in \O(k)$ define an analytic germ on $X(p;-1,q)$, $\gcd(p,q)=1$, 
such that $x \nmid f$. Then there exist $g \in \C\{x,y\}$ with $x \nmid g$
such that $g(x^p,x^q y) = x^{qr} f(x,y)$.
 Moreover, $f$ is reduced (resp.~irreducible) 
if and only if $g$~is.
\end{lemma}

\begin{proof} 
By the discussion after Remark~\ref{weierstrass_remark} one can assume $f\in \C\{x\}[y]$ as in~\eqref{special_form2}.
We have  $-i+qj\equiv qr\equiv k  \mod p$ for all $i,j$ so  $p| (i+q(r-j))$ and $i+q(r-j)>0$. Consider
$$
g(x,y) =y^r+ \sum_{i> 0, \ j<r} a_{ij} x^{\frac{i+q(r-j)}{p}} y^j \ \in \ \C\{x\}[y],
$$
$$
g(x^p,x^qy) =x^{qr}y^r+ \sum_{i> 0, \ j<r} a_{ij} x^{i+qr} y^j=x^{qr}\left(y^r+ \sum_{i> 0, \ j<r} a_{ij} x^i y^j\right).
$$
Note that the strict transform passes only through the origin of the first chart.
\end{proof}

The following Proposition \ref{proppick}  will be useful to give a generalization of Remark~\ref{condicionesclasico}. 

Before we state the result we need some notation. Given $r,p,q\in \Z_{>0}$ we
define the following combinatorial number 
which generalizes $\binom{d}{2}$:
\begin{equation}\label{eq-delta-comb}
\delta^{(p,q)}_{r}:= \frac{r (qr-p-q+1)}{2p}.
\end{equation}
Note that $\binom{d}{2} = \delta^{(1,1)}_d$.

\begin{prop}\label{proppick}
 Let be $p,q,a,r \in \Z_{>0}$ with $\gcd(p,q)=1$ and $r_1=r+pa$. Consider the following cardinal,
$$A^{(p,q)}_r:= \#\{ (i,j)\in \Z^2|\  pi+qj\leq qr;\ i,j\geq 1 \}.$$ Then,

\begin{enumerate}
 \item[1)] If $r=pa$, one has $\delta_{r}^{(p,q)}=A_{r}^{(p,q)}$.
 \item[2)] The following equalities hold:
 \begin{equation}\label{3.2.1}
 \delta^{(p,q)}_{r_{1}}-\delta_{r}^{(p,q)}=\delta_{r_{1}-r}^{(p,q)}+aqr,
 \end{equation}
 \begin{equation}\label{3.2.2}
   A_{r_{1}}^{(p,q)}-A_{r}^{(p,q)}=A_{r_{1}-r}^{(p,q)}+aqr.
 \end{equation}
 \item[3)] The difference $A^{(p,q)}_r-\delta^{(p,q)}_r$ only depends on $r$
modulo $p$. 
\end{enumerate}
\end{prop}

\begin{proof}

\hspace{1cm}
 \begin{enumerate}

\item [1)] To prove this fact it is enough to apply Pick's Theorem (see for example \cite[\S$2.6$]{Beck07}) noticing that the number of points on the diagonal without counting the ones in the axes is $a-1$.

Finally one gets,
$$A_{pa}^{(p,q)}=\frac{a(pqa-p-q+1)}{2}=\delta^{(p,q)}_{pa}.$$
\item[2)] Proving equation (\ref{3.2.1}) is a simple and direct computation. 
To prove equation (\ref{3.2.2}), let us describe $A^{(p,q)}_{r_{1}}$, $A^{(p,q)}_r$ and $ A^{(p,q)}_{r_{1}-r}$:

\begin{equation*}
 A^{(p,q)}_{r_{1}}= \# \{ (i,j)\in\Z^2|\ pi+qj\leq qr+pqa;\ i,j\geq 1 \}.
\end{equation*}
\begin{figure}[h]
\begin{center}
 \includegraphics[scale=0.7]{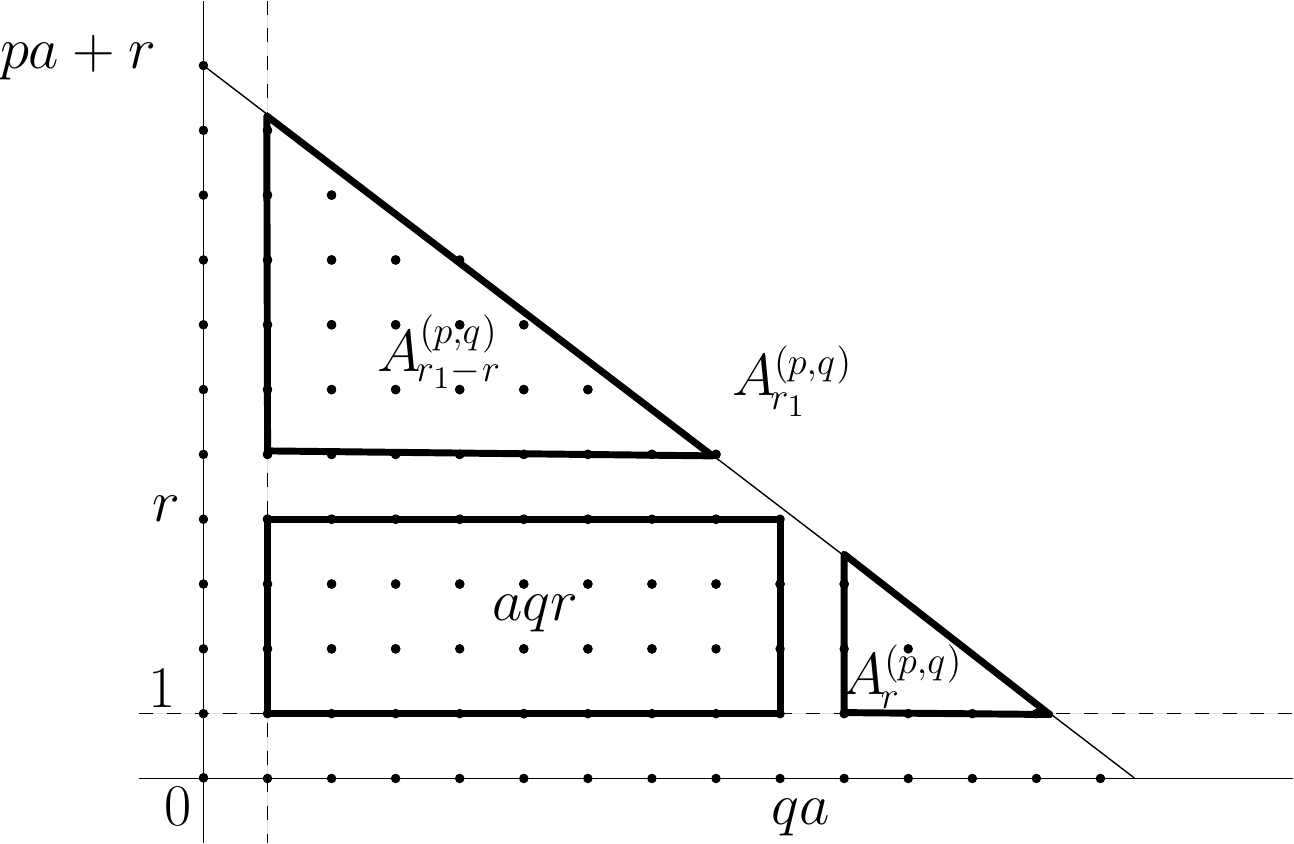}
\caption{}
\label{pick2fig}
\end{center}
\end{figure}
\begin{eqnarray*}
 A^{(p,q)}_r &= &\# \{ (i,j)\in\Z^2|\ pi+qj\leq qr;\ i,j\geq 1 \}\\
&=&\# \{ (i,j)\in\Z^2|\ pi+qj\leq qr+apq-apq;\ i,j\geq 1 \}\\
&=&\# \{ (i,j)\in\Z^2|\  p(i+aq)+qj\leq qr_1;\ i\geq 1, \ j\geq1 \}\\
&=&\# \{ (i,j)\in\Z^2|\ pi+qj\leq qr_{1};\ i\geq aq+1, \ j\geq1 \}.
\end{eqnarray*}
\begin{eqnarray*}
 A^{(p,q)}_{r_{1}-r} &= &\# \{ (i,j)\in\Z^2|\  pi+qj\leq qr_{1}-qr;\ i,j\geq 1 \}\\
&=&\# \{ (i,j)\in\Z^2|\  pi+q(j+r)\leq qr_{1};\ i,j\geq 1 \}\\
&=&\# \{ (i,j)\in\Z^2|\  p+qj\leq qr_{1};\ i\geq 1, \ j\geq r+1 \}.
\end{eqnarray*}

Using the decomposition shown in Figure~\ref{pick2fig} the claim follows. 

\item[3)] Subtracting equations (\ref{3.2.1}) and (\ref{3.2.2}) the result
holds.
\end{enumerate}
\end{proof}

\begin{dfn}\label{defDELTA}
Let $k\geq 0$ and $P\in X(p;-1,q)=X$.
The \emph{ $\Diff_P(k)$-invariant} of $X$  is defined as follows
$$\Diff_P(k):= A^{(p,q)}_r-\delta^{(p,q)}_r,$$
where $r=q^{-1}k \mod p$.
\end{dfn}

As a result of Proposition~\ref{proppick} one has the following result.

\begin{thm}\label{teoremadeltaext}
Let $f_1,f_2\in \O(k)$ be two germs at $[0]\in X(d;a,b)$.
Then,
$$
\K_0(f_1)-\K_0(f_2)=\delta_0(f_1)-\delta_0(f_2).
$$
\end{thm}

\begin{proof}
By Remark~\ref{weierstrass_remark} and the discussion after it, we can assume that 
\begin{equation*}
f_\ell(x,y) = y^{r_\ell} + \sum_{i> 0\leq j<r_\ell} a_{ij} x^i y^j \ \in \ \C\{x\}[y].
\end{equation*}
in $X(p;-1,q)$ ($p=d, q\equiv -ba^{-1} \mod d$).
Consider $g_1\in \C\{x,y\}$ the reduced germ obtained after applying Lemma~\ref{lema1b} to~$f_1$. 
Denote by $\pi_{(p,q)}$ the blowing-up at the origin. Note that $\nu_{p,q}(g_1)=qr_1$,
and thus $\delta^{(p,q)}_{r_1}=\delta_{\pi_{(p,q)}}(g_1)$ (see~\eqref{eq-delta-comb} and~\eqref{eq-deltawpi}).

Consider the form $\omega:=\phi\frac{dx\wedge dy}{g_1}$, $\phi\in\C\{x,y\}$ and let us calculate 
the local equations for the pull-back of $\omega$ after blowing-up the origin on $\C^2$,
\begin{equation}
\label{eq-blowup-w}
\phi\frac{dx\wedge dy}{g_1}\overset{\pi_{(p,q)}}{\longleftarrow} x^{\nu_{\phi}+p+q-1-qr_1}h\frac{dx\wedge dy}{f_1}.
\end{equation}

Using the definitions of ${\mathcal M}^{\nul}_{g_1}$ and ${\mathcal M}^{\nul}_{f_1}$ (see Definition~\ref{modulitos1})
this implies that
$$\phi\in {\mathcal M}^{\nul}_{g_1} \Leftrightarrow h \in {\mathcal M}^{\nul}_{f_1} \
\text{ and } \ \nu_{\phi}+p+q-1-qr_1\geq 0.$$
Therefore $\phi(x,y) \mapsto \phi(x^p,x^qy)$ induces an isomorphism
$$
{\mathcal M}^{\nul}_{g_1}\cong \mathcal A_{r_1}^{(p,q)}\cap {\mathcal M}^{\nul}_{f_1},
$$
where $\mathcal A_{r_1}^{(p,q)}:=\{h\in \C\{x,y\}\mid \ord_h+p+q-1-qr_1\geq 0\}$ and $\ord_h$ is the $(p,q)$-order of $h$.
Since $\dim_\C \frac{\C\{x,y\}}{\mathcal A_{r_1}^{(p,q)}}=A_{r_1}^{(p,q)}$, one obtains
\begin{equation}\label{eq-kgAkf}
 \K_0(g_1)=A_{r_1}^{(p,q)} + \K_0(f_1)
\end{equation}

On the other hand (see Remark~\ref{condicionesclasico} and Theorem~\ref{thm-delta}), 
\begin{equation}\label{eq-kgAkf2}
 \K_0(g_1)=\delta_0(g_1)=\delta_{\pi_{(p,q)}}(f)+\delta_0(f_1)=\delta^{(p,q)}_{r_1}+\delta_0(f_1).
\end{equation}
Therefore, from \eqref{eq-kgAkf} and \eqref{eq-kgAkf2}, 
\begin{equation}\label{etiqueta1}
 \K_0(f_1)=\delta_0(f_1)+\delta^{(p,q)}_{r_1}-A^{(p,q)}_{r_1}.
\end{equation}

Following a similar procedure we get,
\begin{equation}\label{etiqueta2}
\K_0(f_2)=\delta_0(f_2)+\delta^{(p,q)}_{r_2}-A^{(p,q)}_{r_2}.
\end{equation}
Notice that $k\equiv qr_1\equiv qr_2 \mod p$, which implies $r_1\equiv r_2 \mod p$ since $p$ and $q$ are coprime. Therefore
by Proposition~\ref{proppick} 
$$A^{(p,q)}_{r_1}-\delta^{(p,q)}_{r_1}=A^{(p,q)}_{r_{2}}-\delta^{(p,q)}_{r_{2}},$$
and finally from (\ref{etiqueta1}) and (\ref{etiqueta2}) it can be concluded that
$$
\K_0(f_1)-\K_0(f_2)= \delta_0(f_1)-\delta_0(f_2).
\vspace{-0.5cm}
$$ 
\end{proof}

\begin{proof}[{Proof of Theorem~\ref{thmDELTAJ}}]
The result follows directly from equation~\eqref{etiqueta1}.
\end{proof}

\begin{remark}\label{condicionesfuncion}
If $(f,[0])$ is a function germ on $X=X(d;a,b)$, from
Proposition~\ref{proppick}(1) and Theorem~\ref{thmDELTAJ}, one has
$$\K_P(f)=\delta_P(f).$$
In particular if $P$ is a smooth point of $X$, this generalizes
Remark~\ref{condicionesclasico}.
\end{remark}

\section{An introductory example}\label{secex}

Let us start this section with one basic illustrative example. Let us compute
the number of solutions $(a,b,c)\in\Z_{\geq 0}^3$ of the equation $$a\w_0+b\w_1+c\w_2=
k\bar{\w}$$ with $\w_0, \w_1, \w_2 \in \Z_{>0}$ and $k \in \Z_{\geq 0}$ fixed, or equivalently, the number
of monomials in $\O_{\P^2_{\w}}$ of quasi-homogeneous degree $k\bw$.
% (recall the notation used in \S\ref{secgen})
This number will be denoted by
$\Eh_\w(k\bar{\w})$ (recall \eqref{cardinal}). Notice that this is equivalent to
computing the number of non-negative integer solutions $(a,b,c)$ to
$a\w_0+b\w_1=(k\w_{01}-c)\w_2$ (with $w_{ij}:=w_iw_j$), which can be achieved by considering the following sets:
\begin{align*}
\tilde{A}:= & \left\{ (a,b)\in \Z_{>0}^2 \mid a\w_0+b\w_1=\alpha\w_2,\ \alpha=0,\ldots, k\w_{01}\right\}, \\
\tilde{B}:= & \left\{ (a,0)\in \Z_{>0}^2 \mid a\w_0=\alpha\w_2, \ \alpha=0,\ldots, k\w_{01} \right\} \cup \\
& \left\{ (0,b) \in \Z_{>0}^2 \mid b\w_1= \alpha \w_2, \ \alpha=0,\ldots, k\w_{01} \right\}.
\end{align*}
If we denote by $A= \#\tilde{A}$ and $B= \#\tilde{B}$, one has $\Eh_\w(k\bar{\w})=A+B+1$.
To compute $A$ take two integers $n_0,\ n_1 $ such that $n_0\w_0+n_1\w_1=1$ with $n_1>0$ and $n_0\leq 0$
(this can always be done since the weights are pairwise coprime). There exists a positive integer $\lambda$
satisfying $a = n_0 \alpha \w_2 + \lambda \w_1$ and $-\frac{n_0\alpha\w_2}{\w_1}<\lambda<\frac{n_1\alpha\w_2}{\w_0}$.
This justifies the following definition: (see Figure~\ref{numeropuntosbn})
$$
A_{\alpha}:=\#\left\{ \lambda \in \Z_{>0} \ |-\frac{n_0\alpha\w_2}{\w_1}<\lambda<\frac{n_1\alpha\w_2}{\w_0} \right\}.
$$

\begin{figure}
\begin{center}
 \includegraphics[scale=1]{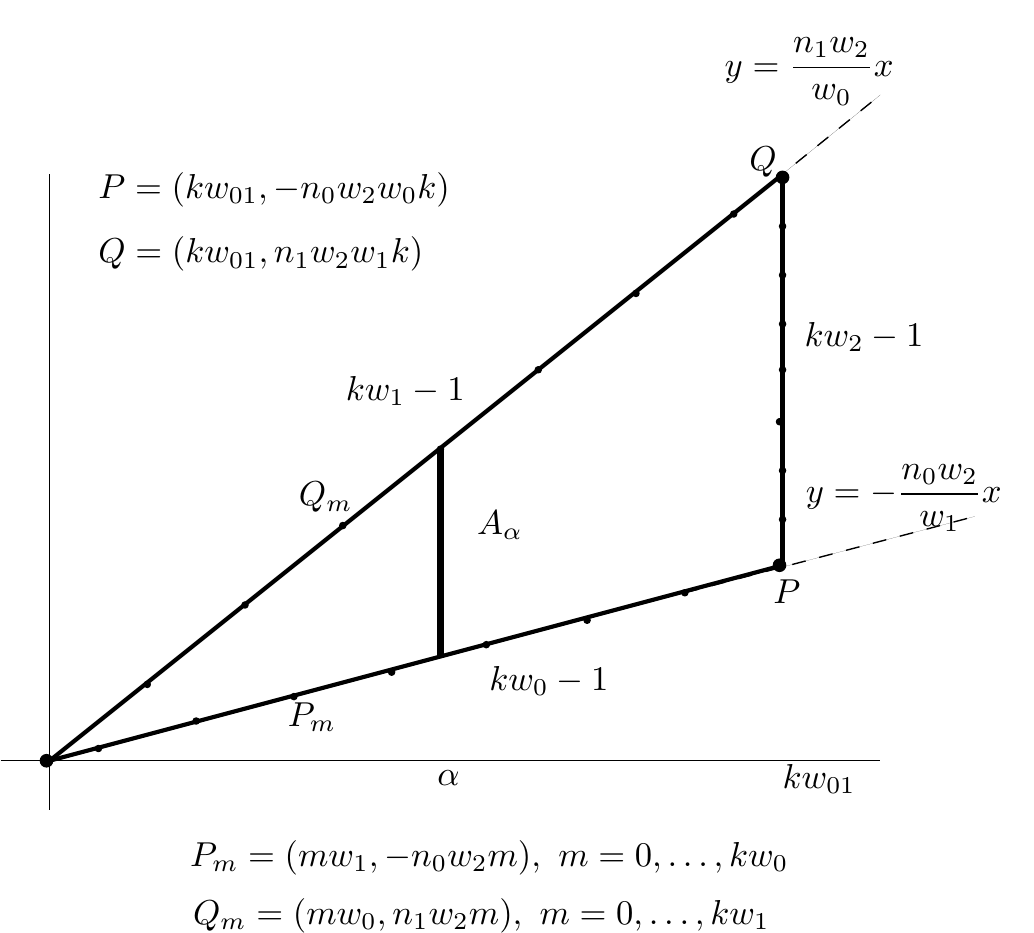}
\caption{}
\label{numeropuntosbn}
\end{center}
\end{figure}

Note that by virtue of Pick's theorem the area of the triangle is equal to the number of natural points in its interior $I$ plus one half the number of points in the boundary minus one. The area of the triangle equals $\frac{k^2\bar{\omega}}{2}$, thus 
$$\frac{k^2\bar{\omega}}{2}=I+\frac{k|\w|}{2}-1,$$
which implies $$A=I+(k\w_2-1)=\left(\frac{k^2\bar{\w}}{2}-\frac{k|\w|}{2}+1\right)+(k\w_2-1)=\frac{1}{2}(k^2\bar{\w}-k|\w|)+k\w_2.$$
It is easy to check that $B=k\w_0+k\w_1$, then we have
$$\Eh_\w(k\bar{\w})=\frac{1}{2}k\left(k\bar{\w}+|\w| \right)+1.$$
It is known that the genus of a smooth curve on $\P^2_{\w}$ of degree $d$
transversal w.r.t.~the axes is $$g_{\w,d}= \frac{d(d-|\w|)}{2\bar{\w}}+1.$$ We
want to find $d$ such that $\Eh_\w(k\bar{\w})=g_{\w,d}$. To do that it is enough
to solve the equation 
$$\frac{1}{2}k\left(k\bar{\w}+|\w| \right)+1=\frac{d(d-|\w|)}{2\bar{\w}}+1.$$
One finally gets that
$$\Eh_\w(k\bar{\w})=g_{\w,|w|+k\bar{\w}}.$$

The rest of this paper deals with the extension of this example when $d$ is not
necessarily a multiple of $\bar w$.

\section{Proof of the main results}\label{Goretti}

\begin{proof}[{Proof of Theorem~\ref{AdjFor}}]
We will prove the equivalent formula
$$
\Eh_\w(d-|\w|)=g_{\w,d}-\sum_{P\in\Sing(\P^2_w)}\Diff_P(d).
$$
From the definitions note that
$$
\Eh_\w(d-|\w|)=p_{\{w_0,w_1,w_2\}}(d-|\w|)=:p_{\{\w\}}(d-|\w|).
$$
Fix a point $P\in\Sing(\P^2_w)$ and describe for simplicity the local singularity as 
$X(w_i;w_{i+1},w_{i+2})=X(\w_i;-1,q_i)$, where $q_i:=-w_{i+1}^{-1}w_{i+2} \mod w_i$, 
for $i=0,1,2$ (indices are considered modulo 3). Define $r_i:=w_{i+2}^{-1}d \mod w_i$, then 
$$
A_{r_i}^{(w_i,q_i)}=p_{\{w_i,q_i,1\}}(q_ir_i-w_i-q_i).
$$
On the one hand from~\eqref{ecpw} and a direct computation one obtains
\begin{eqnarray*}
\Eh_\w(d-|w|)-g_{\w,d}&=&-1+\poly_{\{w\}}(0)+\sum_{i=0}^2 s_{|w|-d}(w_i,w_{i+1};w_{i+2}).
\end{eqnarray*}

By Definition~\ref{FDS} and Corollary~\ref{radrec} one obtains
\begin{equation}\label{Gor1}
\Eh_\w(d-|w|)-g_{\w,d}=\sum_{i=0}^2\left(s_{|w|-d}(w_{i+1},w_{i+2};w_i)-s_{0}(w_{i+1},w_{i+2};w_i) \right).
\end{equation}
On the other hand, from \eqref{ecpw1}, and straightforward computations one obtains
\begin{eqnarray}\label{Gor2}
\delta_{r_i}^{(w_i,q_i)}-A_{r_i}^{(w_i,q_i)}&=&\frac{w_i+q_i}{2w_iq_i}-\poly_{\{w_i,q_i,1\}}(0)\nonumber\\
&-&\left(s_{w_i+q_i-q_ir_i}(w_i,1;q_i)+s_{w_i+q_i-q_ir_i}(q_i,1;w_i)\right),
\end{eqnarray}
with
$$s_{w_i+q_i-q_ir_i}(q_i,1;w_i)=\frac{1}{w_i}\sum_{k=1}^{w_i-1}\frac{1}{(1-\xi_{w_i}^{kq_i})(1-\xi_{w_i}^{k})
\xi_{w_i}^{k(q_ir_i-w_i-q_i)}},$$
and 
\begin{eqnarray*}
s_{w_i+q_i-q_ir_i}(w_i,1;q_i)&=&\frac{1}{q_i}\sum_{k=1}^{q_i-1}\frac{1}{(1-\xi_{q_i}^{kw_i})(1-\xi_{q_i}^{k})
\xi_{q_i}^{k(q_ir_i-w_i-q_i)}}\\
&=&-\frac{1}{q_i}\sum_{k=1}^{q_i-1}\frac{1}{(1-\xi_{q_i}^{-kw_i})(1-\xi_{q_i}^{k})},
\end{eqnarray*}
which implies using Proposition~\ref{prop3}
\begin{equation}\label{eqaux}
 s_{w_i+q_i-q_ir_i}(w_i,1;q_i)= s_{w_i}(w_i,1;q_i)= s(-w_i,q_i)-\frac{q_i-1}{4q_i}.
\end{equation}

Since by hypothesis $q_i=-w_{i+1}^{-1}w_{i+2} \mod w_i$ and  $r_i=w_{i+2}^{-1}d \ \mod w_i$, one obtains
\begin{large}
\begin{equation}
\array{c}
s_{w_i+q_i-q_ir_i}(q_i,1;w_i)=\frac{1}{w_i}\sum_{\ell=1}^{w_i-1}\frac{1}{(1-\xi_{w_i}^{\ell q_i})(1-\xi_{w_i}^{\ell})\xi_{w_i}^{\ell(q_ir_i-w_i-q_i)}}\\
\\
=\frac{1}{w_i}\sum_{\ell=1}^{w_i-1}\frac{1}{(1-\xi_{w_i}^{\ell(-w_{i+1}^{-1}w_{i+2})})(1-\xi_{w_i}^{\ell})\xi_{w_i}^{\ell(-w_{i+1}^{-1}d+w_{i+1}^{-1}w_{i+2})}}\\
\\
\overset{\ell=-w_{i+1}\bar \ell}{=} \frac{1}{w_i}\sum_{\bar \ell=1}^{w_i-1}
\frac{1}{(1-\xi_{w_i}^{\bar \ell w_{i+2}})(1-\xi_{w_i}^{-\bar \ell w_{i+1}})\xi_{w_i}^{\bar \ell (d-w_{i+2})}}\\
\\
=-\frac{1}{w_i}\sum_{\bar \ell =1}^{w_i-1}
\frac{1}{(1-\xi_{w_i}^{\bar \ell w_{i+2}})(1-\xi_{w_i}^{\bar \ell w_{i+1}})\xi_{w_i}^{\bar \ell (d-|w|)}}
=-s_{|w|-d}(w_{i+1},w_{i+2};w_i).
\endarray
\end{equation}
\end{large}

Thus 
\begin{equation}\label{Gor3}
 s_{|w|-d}(w_{i+1},w_{i+2};w_i)=-s_{w_i+q_i-q_ir_i}(q_i,1;w_i),
\end{equation}
for $i=0,1,2$.

Using~\eqref{Gor1}, \eqref{Gor2}, and \eqref{Gor3}, it only remains to show
\begin{equation}
 \label{Gor5}
-s_{0}(w_{i+1},w_{i+2};w_i)=\frac{w_i+q_i}{2w_iq_i}-\poly_{\{w_i,1,q_i\}}(0)- s_{w_i+q_i-q_ir_i}(w_i,1;q_i).
\end{equation}
For the left-hand side we use Remark~\ref{remDed}(4) and obtain
$$s_{0}(w_{i+1},w_{i+2};w_i)=-s(-q_i,w_i)+\frac{w_i-1}{4w_i}=s(q_i,w_i)+\frac{w_i-1}{4w_i}.$$
For the right-hand side, using Corollary~\ref{radrec} and (\ref{eqaux}) we have,
\begin{equation*}
\poly_{\{w_i,1,q_i\}}(0) + s_{w_i+q_i-q_ir_i}(w_i,1;q_i) =
1-s_{0}(q_i,1;w_i)-s_{0}(w_i,1;q_i)+ s_{w_i}(w_i,1;q_i),
\end{equation*}
which, by Remark~\ref{remDed}(3) and  (\ref{eqaux}) becomes
\begin{equation*}
\left(1+s(q_i,w_i)-\frac{w_i-1}{4w_i}+s(w_i,q_i)-\frac{q_i-1}{4q_i}\right)+s(-w_i,q_i)-\frac{q_i-1}{4q_i}.
\end{equation*}
Combining these equalities into~\eqref{Gor5} one obtains the result.
\end{proof}

\begin{proof}[{Proof of Theorem~\ref{Adj-like-form}}]
It is enough to apply~\cite[Theorem 5.7]{CMO12}, Theorem~\ref{AdjFor} and 
recall the characterization of $\K_P(f)$ in the proof of
Theorem~\ref{teoremadeltaext} (see~\eqref{etiqueta1}).
\begin{align*}
&g(\CC)= g_{\w,d}-\sum_{P\in\Sing(\CC)}\delta_P(f)\\
&=\underbrace{g_{\w,d}+\sum_{i=0}^2\left(
\delta_{r_i}^{(w_i,q_i)}-A_{r_i}^{(w_i,q_i)}\right)}_{\Eh_\w(d-|w|)}-\underbrace{
\left(\sum_{P\in\Sing(\CC)}\delta_P(f)+\sum_{i=0}^2\left(
\delta_{r_i}^{(w_i,q_i)}-A_{r_i}^{(w_i,q_i)}\right)\right)}_{\sum _{P\in\Sing(\cC)} \K_P(f)}.
\end{align*} 
\end{proof}
\begin{remark}
The second equality in the previous identity always holds and therefore if one
considers $\cC\subset \PP^2_{\w}$ a reduced curve of degree $d$, then (recall Theorem~\ref{thmDELTAJ})
\begin{align*}
\Eh_\w(d-|\w|)&=g_{\w,d}-\sum_{P\in\Sing(\CC)} \delta_P(f)+\sum _{P\in\Sing (\cC)} \K_P(f).
\end{align*}
\end{remark}

Let us see an example of the previous result.
\begin{example}\label{exampadj}
Consider the polygon $\cD_\w:=\{(x,y,z)\in \RR^3\mid w_0x+w_1y+w_2z=1\}$, for $\w=(w_0,w_1,w_2)=(2,3,7)$.
As an example, we want to obtain the Ehrhart quasi-polynomial $\Eh_\w(d)$ for $\cD_\w$. Note that, according to
Theorem~\ref{AdjFor}
$$
\Eh_\w(d)=\frac{1}{84}d^2+\frac{1}{7}d+a_0(d),
$$
where $a_0(d)$ is a rational periodic number of period $\bar \w=42$. Moreover, 
$a_0(d)=1-\left( \sum_{i=0}^{2}\Delta_i(d+12) \right)$, where $\Delta_i$ has period $w_i$ and depends only on the 
singular point $P_i=\{x_j=x_k=0\}$ ($\{i,j,k\}=\{0,1,2\}$) in the weighted projective plane~$\PP^2_\w$.

In order to describe $\Delta_i(d)$ we will introduce some notation. Given a list of rational numbers 
$q_0,\dots,q_{r-1}$ we denote by $[q_0,\dots,q_{r-1}]$ the periodic function $f:\ZZ\to \QQ$ whose period 
is $r$ and such that $f(i)=[q_0,\dots,q_{r-1}]_i=q_i$. Using this notation it is easy to check that 
$$
\Delta_0(d)=\left[ 0,\frac{1}{4}\right]_d \quad \text{and}  \quad
\Delta_1(d)=\left[ 0,\frac{1}{3},\frac{1}{3}\right]_d.
$$
Finally, in order to obtain $\Delta_2(d)$, one needs to compute both $\delta$ and $K$-invariants for the 
singular point $P_2 \in X=X(7;2,3)$.
The following table can be obtained directly:

\begin{table}[h]
\caption{Local invariants at $X(7;2,3)$} % title name of the table
\centering 
\begin{tabular}{|c||c|c|c|c|c|c|c|}
%\begin{array}{|c|c|c|c|c|c|c|c|}
\hline
$d$ & 0 & 1 & 2 & 3 & 4 & 5 & 6 \\
\hline \hline
$\Delta_P$ & 0 & $\nicefrac{2}{7}$ & $\nicefrac{3}{7}$ & $\nicefrac{3}{7}$ & $\nicefrac{2}{7}$ & 0 & $\nicefrac{4}{7}$ \\
\hline
$\delta_P$ & 0 & $\nicefrac{9}{7}$ & $\nicefrac{3}{7}$ & $\nicefrac{3}{7}$ & $\nicefrac{9}{7}$ & 1 & $\nicefrac{4}{7}$ \\
\hline
$\K_P$ & 0 & 1 & 0 & 0 & 1 & 1 & 0 \\
\hline
\text{Branches} & 0 & 2 & 1 & 1 & 2 & 2 & 1 \\
\hline 
\raisebox{-.5ex}{\text{Equation}} & \raisebox{-.5ex}{1} & \raisebox{-.5ex}{$x(x^3+y^2)$} & 
\raisebox{-.5ex}{$x$} & \raisebox{-.5ex}{$y$} & \raisebox{-.5ex}{$x(x+y^3)$} & \raisebox{-.5ex}{$xy$} & 
\raisebox{-.5ex}{$x^3+y^2$} \\[0.5ex]
\hline
%\end{array}
\end{tabular}
\label{tab:inv}
\end{table}

The typical way to obtain the first row is by applying Theorem~\ref{Adj-like-form} to a generic germ $f_d$ in
$\mathcal O_X(d)$. This is how the second and third rows in the previous table were obtained. The last two rows
indicate the local equations and the number of branches of such a generic germ~$f_d\in \mathcal O_X(d)$.

Let us detail the computations for the third column in Table~\ref{tab:inv} (case $d=1$). 
One can write the generic germ $f_1$ in $\mathcal O_X(1)$ as $x(x^3+y^2)$. On the one hand,
a $(2,3)$-blow-up serves as a $\Q$-resolution of $X$, and thus, using Theorem~\ref{thm-delta}, 

$$\delta_P(f_1)=\dfrac{8(8-2-3+7)}{2\cdot7\cdot2\cdot3}+\dfrac{3-1}{2\cdot 3}=\dfrac{9}{7}.$$

For a computation of $\K_P$ one needs to study the quotient $\O_X(3)/\mathcal{M}^{\nul}_{f_1}$. 
Notice that in the present case 
$\O_{X}=\C\{x,y\}^{\qa_7}=\C\{x^7,y^7,x^2y\}$ and $\O_{X}(3)$ is the $\O_{X}$-module generated by $y$ and $x^5$. 
In order to study $\mathcal{M}^{\nul}_{f_1}$, consider a generic form 
$$(ay+bx^5)\dfrac{dx\wedge dy}{f_1}\in \Omega_X^2(\logres \langle f_1\rangle),$$ 
where $a,b\in \O_{X}$ and its pull-back by a resolution of the singularity $X(7;2,3)$. One obtains the following:
\vspace{0.2cm}

\begin{equation}\label{eq-ejt4}
\array{ll}
(ay+bx^5)\ \dfrac{dx \wedge dy}{x(x^3-y^2)} \leftblowup{x=u_1\bar v_1^2}{y=\bar v_1^3,\ {v}_1=
\bar v_1^7} & 3 \bar v_1^3(\tilde a+\tilde b u_1^5\bar v_1^7) 
\dfrac{ \bar v_1^{4}\ du_1 \wedge d\bar v_1}{\bar v_1^8u_1(u_1^3-1)} =\\
&\\
&\dfrac{3}{7}(\tilde a+\tilde bu_1^5v_1) \dfrac{du_1 \wedge dv_1}{v_1u_1(u_1^3-1)}.
\endarray
\end{equation}

Therefore $(ay+bx^5)\notin \mathcal{M}^{\nul}_{f_1}$ iff the function $a\in \O_X$ is a unit.
Hence, by Definition~\ref{condicionesnul},
$$\K_P(f_1)=\dim_{\C} \frac{\O_X(3)}{{\mathcal M}^{\nul}_{f_1}}=\dim_{\C} <y>_\C=1.$$

Finally, $$\Delta_P(1)=\delta_P(f_1)-\K_P(f_1)=\dfrac{2}{7}.$$

The rest of values in Table~\ref{tab:inv} can be computed analogously. Hence one obtains:
$$
\Delta_2(d):=\left[ 0, \frac{2}{7}, \frac{3}{7}, \frac{3}{7}, \frac{2}{7}, 0, \frac{4}{7} \right]_d,
$$
and thus
$$
\Eh_\w(d)=\frac{1}{84}d^2+\frac{1}{7}d+\left( 1-\left[ 0,\frac{1}{4}\right]_d - 
\left[ 0,\frac{1}{3},\frac{1}{3}\right]_d-
\left[0,\frac{4}{7},0,\frac{2}{7},\frac{3}{7},\frac{3}{7},\frac{2}{7}\right]_d\right).
$$
For instance, if one wants to obtain $\Eh_\w(54)$, note that 
$\left[ 0,\frac{1}{4}\right]_{54}=\left[ 0,\frac{1}{4}\right]_{0}=0$, 
$\left[ 0,\frac{1}{3},\frac{1}{3}\right]_{54}=\left[ 0,\frac{1}{3},\frac{1}{3}\right]_0=0$, and 
$\left[0,\frac{4}{7},0,\frac{2}{7},\frac{3}{7},\frac{3}{7},\frac{2}{7}\right]_{54}=
\left[0,\frac{4}{7},0,\frac{2}{7},\frac{3}{7},\frac{3}{7},\frac{2}{7}\right]_{5}=\frac{3}{7}$.
Thus
$$
\Eh_\w(54)=\frac{1}{84}54^2+\frac{1}{7}54+\left(1-\frac{3}{7}\right)=43.
$$
\end{example}

\bibliographystyle{plain}
%\bibliography{counting}
% \renewcommand{\bibname}{BIBLIOGRAPHY}
\def\cprime{$'$}

\end{document}